\documentclass[12pt]{amsart}
 \usepackage{mathtools}
\mathtoolsset{showonlyrefs,showmanualtags}

\usepackage{amsrefs}  
 \usepackage[top=1.5in, bottom=1.5in, left=1.25in, right=1.25in]	{geometry}
\usepackage[all]{xy}
\usepackage{amsmath}
\usepackage{amssymb}
\usepackage{amsthm}
\usepackage{amscd}

\newcommand{\n}{\mathcal{N}}

\newcommand{\HH}{\mathcal{H}}

\newcommand{\U}{\mathcal{U}}

\newcommand{\PP}{\mathbb  P}

\newcommand{\Z}{\mathbb  Z}
\newcommand{\I}{\mathbb  I}
\newcommand{\RR}{\mathbb  R}

\newcommand{\C}{\mathbb  C}

\numberwithin{equation}{section}

\newtheorem{theorem}[equation]{Theorem}
\newtheorem{definition}[equation]{Definition}
\newtheorem{proposition}[equation]{Proposition}
\newtheorem{cor}[equation]{Corollary}
\newtheorem{lemma}[equation]{Lemma}
\newtheorem{problem}[equation]{Problem}
\newtheorem{remark}[equation]{Remark}

\usepackage{hyperref} 
\hypersetup{
    colorlinks=true,       
    linkcolor=blue,          
    citecolor=magenta,        
    filecolor=magenta,      
    urlcolor=cyan           
}

\begin{document}
\title{On Convergence of Oscillatory Ergodic Hilbert Transforms}
\author{Ben Krause}
\address{
Department of Mathematics
The University of British Columbia \\
1984 Mathematics Road
Vancouver, B.C.
Canada V6T 1Z2}
\email{benkrause@math.ubc.ca}
\thanks{Research supported in part by an NSF Postdoctoral Research Fellowship. This material is based upon work supported by the National Science Foundation under Grant No. DMS-1440140, while the author was in residence at the Mathematical Sciences Research Institute in Berkeley, California, during the Spring 2017 Semester. }

\author{Michael Lacey}
\address{
School of Mathematics
Georgia Institute of Technology \\
686 Cherry Street
Atlanta, GA 30332-0160}
\email{lacey@math.gatech.edu}
\thanks{Research supported in part by grant   National Science Foundation grant DMS-1600693, and by Australian Research Council grant DP160100153.  This material is based upon work supported by the National Science Foundation under Grant No. DMS-1440140, while the author was in residence at the Mathematical Sciences Research Institute in Berkeley, California, during the Spring 2017 Semester. }
\author{M\'{a}t\'{e} Wierdl}
\address{University of Memphis\\
  Department of Math Sciences\\
  373 Dunn Hall \\
  Memphis, TN 38152}
\email{mwierdl@memphis.edu}

\date{\today}

\begin{abstract}
We introduce  sufficient conditions on discrete singular integral operators for their maximal truncations to
satisfy a \emph{sparse bound}. The latter imply a range of quantitative weighted inequalities, which are new.  
As an application, we prove the following ergodic theorem:
let $p(t)$ be a Hardy field function which grows ``super-linearly" and stays ``sufficiently far" from polynomials. We show that for each measure-preserving system, $(X,\Sigma,\mu,\tau)$, with $\tau$ a measure-preserving $\Z$-action, the modulated one-sided ergodic Hilbert transform
\[ \sum_{n=1}^\infty \frac{e^{2\pi i p(n)}}{n} \tau^n f(x) \]
converges $\mu$-a.e.\ for each $f \in L^r(X), \ 1 \leq r < \infty$.
This affirmatively answers a question of J. Rosenblatt \cite{MR939919}.   

In the second part of the paper, we  establish almost sure sparse bounds for  \emph{random} one-sided ergodic Hilbert,
\[ \sum_{n=1}^\infty \frac{X_n}{n} \tau^n f(x), \]
where $\{ X_n \}$ are uniformly bounded, independent, and mean-zero random variables.

\end{abstract}

\maketitle

\section{introduction}


Our subject is discrete Harmonic Analysis.  We give sufficient conditions for  maximal truncations of  discrete singular integral operators to have
sparse bounds. Our argument has as its antecedents the Fefferman \cite{MR0257819} and Christ \cite{MR951506} $TT^*$ approach to proving weak-type $(1,1)$ bounds for rough singular integral operators on Euclidean space; this approach has already appeared in the discrete context in the work of LaVictoire \cite{MR2576702} where $\ell^1 \to \ell^{1,\infty}$ endpoint estimates for certain (random) maximal functions are proven, and in Urban and Zienkiewicz \cite{MR2318621}
and Mirek \cite{MR3421994}
where endpoint estimates for (deterministic) maximal functions taken over ``thin'' subsets of the integers are established.

We prove \emph{sparse bounds}, which in turn easily imply quantitative weighted bounds, which are novel in this context. 
Moreover, we address maximal truncations, which is also new.  

Sparse bounds are a relatively new topic, and 
we set some notation to describe sparse bounds.  
Say that $ I\subset \mathbb Z $ is an interval if $ I= [a,b] \cap \mathbb Z $ for $ a, b \in \mathbb R $. 
Define 
\begin{equation*}
\langle f \rangle _{I,r} = \Bigl[ \frac 1 {\lvert  I\rvert } \sum_{x\in I} \lvert  f (x)\rvert ^{r}  \Bigr] ^{1/r}.  
\end{equation*}
If $ r=1$, we will frequently write $ \langle f \rangle _{I,r} = \langle f \rangle_I$. 
We say that a collection of intervals $ \mathcal S$ is \emph{sparse} if for all $ S \in \mathcal  S$ there is a set $ E_S \subset S$ so that the collection of sets $ \{E_S \;:\; S\in \mathcal S\}$ are disjoint, and $ \lvert  E_S\rvert \geq \tfrac 14 \lvert  S\rvert  $ for all $ S\in \mathcal S$.  Define \emph{sparse bilinear forms} by 
\begin{equation}\label{e:Lam}
\Lambda _{\mathcal S, r,s} (f,g) = \sum_{S\in \mathcal S} \lvert  S\rvert \langle f \rangle _{S,r} \langle g \rangle _{S,s}, 
\qquad 1\leq r, s < \infty . 
\end{equation}
Given a (sub)linear operator $ T$, we set $ \lVert T  \;:\; (r,s)\rVert$ to be the smallest constant $ C$ so that for all finitely supported functions $ f,g$ there holds 
\begin{equation}\label{e:sparseDef}
\lvert  \langle T f,g \rangle\rvert  \leq C \sup \Lambda _{\mathcal S, r,s} (f,g), 
\end{equation}
where the supremum is over all sparse forms. Sparse operators are  positive localized operators, hence their mapping properties are very easy to analyze.  The following theorem is a remarkable refinement of the familiar fact that the Hilbert transform is weakly bounded.  The theorem below implies all the standard weighted inequalities, as is explained in the references.  

\begin{theorem}\label{t:HilbertSparse}\cites{150105818,MR3521084}
We have $ \lVert  \mathcal H ^{\mathbb Z ,\ast} \;:\;   (1,1)\rVert < \infty $, where 
\begin{equation*}
 \mathcal H ^{\mathbb Z ,\ast} f (x) = \sup _{N} \Bigl\lvert \sum_{ n \;:\; \lvert  n\rvert>N } \frac {f (x-n)}n  \Bigr\rvert. 
\end{equation*}
\end{theorem}

Our main theorem of this section concerns the maximal truncations 
\begin{align}\label{e:Ha}
 \HH_a ^{\mathbb Z ,\ast } f := \sup_{N \geq 1} \left| \sum_{n=N}^\infty \frac{a(n)}{n} f(x-n) \right| , 
\end{align}
where $ a \;:\; \mathbb Z \mapsto \{ z \in \mathbb C \;:\; \lvert  z\rvert\leq 1 \}$.  
Hypotheses on $a$ are of course required, and best expressed in terms of 
\begin{equation}\label{e:muj}
\mu_j = \mu^a_j := \sum_{2^{j-1} < n \leq 2^j} \frac{a(n)}{n} \delta_n.
\end{equation}
Let $\tilde{g}(x) := \overline{g}(-x)$ denote complex conjugation and reflection about the origin.

Our main theorem (below) is proved in the following section.

\begin{theorem}\label{t:a} Suppose that we have the inequalities below, valid for some $  \epsilon >0$, 
$ C\geq 1$, and all  $ k  > C $.    
\begin{align}
\label{diagcorr}
 |\mu_j * \tilde{\mu_j}(x)|&  \lesssim 2^{-(1+ \epsilon) j},  \qquad x\neq 0 , 
 \\  \label{e:off}
\lVert  \mu_j * \tilde{\mu_k}\rVert _{\infty } & \lesssim 2^{-\epsilon k - j} 
\qquad      k +C\leq j 
\end{align}
Then, we have the inequalities 
\begin{equation}\label{e:Hrs}
\lVert \mathcal H ^{\mathbb Z ,\ast } _a  \;:\; (1,r)\rVert \lesssim \frac 1 {r-1}, \qquad 1< r < 2.  
\end{equation}
\end{theorem}

We stress that this result implies immediately not just the new weak (1,1) inequality, but a range of quantitative weighted inequalities, which are also new in this context.  See \cites{CCPO,MR3531367,160901564} for these details, as well as  more background and history of sparse bounds. 
Our theorem has antecedents in the works \cites{160901564,160908701,CCPO}, which include the discrete setting, 
a sparse bound for maximal truncations, and the outline of a general theory of sparse bounds for ``rough'' singular integrals.  
Sparse quickly imply weighted inequalities, which are new in the $ \ell ^{p} $ setting. 

As an application of this maximal theory, we are able to prove pointwise convergence for a class of \emph{modulated} -- that is, oscillatory -- (one-sided) ergodic Hilbert transforms, which we now proceed to describe:

Let $(X,\Sigma,\mu,\tau)$ be a measure-preserving system, i.e.\ a non-atomic $\sigma$-finite measure space, with $\tau$ a measure-preserving $\Z$ action. A celebrated result due to Cotlar \cite{MR0084632} concerns the almost-everywhere convergence of the ergodic Hilbert transform,
\[ H f:= \sum_{n \neq 0} \frac{ \tau^nf}{n}.\]
In particular, Cotlar \cite{MR0084632} established the following theorem.
\begin{theorem}
For any measure-preserving system, and any $f \in L^r(X), \ 1\leq r < \infty$, $Hf$ converges $\mu$-a.e.
\end{theorem}
Since Cotlar's result, Calder\'{o}n \cite{MR0227354} showed how to transfer analogous convergence results for the real-variable Hilbert transform to the ergodic setting, offering another proof; Petersen \cite{MR691275} has since offered an especially direct proof, using clever covering arguments. In all instances, the cancellation condition
\[ \lim_{N \to \infty} \sum_{0<|n| \leq N} \frac{1}{n} = 0 \]
played a crucial role in the arguments.

Prior to Cotlar's result, a one-sided variant of the ergodic Hilbert transform was introduced by Izumi \cite{MR0000906}:
\[ H_1 f  := \sum_{n=1}^\infty \frac{1}{n} \tau^n f.\]
 Izumi conjectured that for $f \in L^2(X)$ with $\int f = 0$, $H_1f$ would converge almost everywhere.
Unfortunately, Halmos \cite{MR0030708} proved that on any (non-atomic) probability space there always exists mean-zero $f \in L^2(X)$ for which $H_1f$ fails to converge even in the $L^2$ norm. In fact, Dowker and Erd\"{o}s \cite{MR0102584} exhibited mean-zero $f \in L^\infty(X)$ so that 
\[ H_1^* f := \sup_{N \geq 1} \left| \sum_{n=1}^N \frac{1}{n} \tau^n f \right| = + \infty \]
almost everywhere. 

Our Theorem \ref{t:a} will allow us to prove pointwise convergence for ``twisted'' variants of the ergodic Hilbert transform,
\[ \HH_p f := \sum_{n=1}^\infty \frac{e(p(n))}{n} \tau^n f,\]
in the range $1 \leq p < \infty$.
Here and throughout we let $e(t):=e^{2\pi i t}$ denote the complex exponential, and $p(t)$ is taken to be a real-valued Hardy-field function, which grows super-linearly, in a quantifiable way, and stays sufficiently far from the class of polynomials. Good examples of such functions are fractional monomials, $p(t) := t^c$, for non-integer $c > 1$. The presence of the phase $n \mapsto p(n)$ introduces an element of ``randomness'' into the sequence $\{ \frac{e(p(n))}{n} \}$. Indeed, although this sequence is not absolutely summable, a brief argument involving summation by parts and van der Corput's lemma (see below) shows that 
\[ \sum_{n=1}^\infty \frac{e(p(n))}{n} \]
converges conditionally.

Our main result in this direction is the following theorem.
\begin{theorem}\label{main}
Suppose that $p$ is an ``admissible" Hardy field function. Then for any $f \in L^r(X), \ 1 \leq r < \infty$, $\HH_pf$ converges $\mu$-a.e.
\end{theorem}
We defer the definition of ``admissibile" to our section below on Hardy field functions. 

We will prove Theorem \ref{main} by establishing pointwise convergence for simple functions $f: X \to \C$ and by noting the following proposition, which is an immediate consequence of Calder\'{o}n's transference principle and our Theorem \ref{t:a}.

\begin{theorem}\label{mainP}
For $p$ as above, the maximally truncated modulated Hilbert transform,
\[ \HH_p^*f := \sup_{N \geq 1} \left| \sum_{n=N}^{\infty} \frac{e(p(n))}{n} \tau^nf \right| \]
is bounded on $L^r(X)$ for $1 < r < \infty$, and maps $L^{1}(X) \to L^{1,\infty}(X)$.
\end{theorem}

In particular, we are able to give an elegant answer to  the following question of Rosenblatt \cite{MR939919}.    

\begin{problem}\label{Ros}
Does there exist a sequence $\{ c_n\}$ with $\sum_{n=1}^\infty |c_n| = \infty$ such that for $\tau : X \to X$ a measure-preserving transformation, and $f \in L^1(X)$, the series $\sum_{n=1}^\infty c_n \tau^n f(x)$ converges for almost every $x$?
\end{problem}

See however the earlier solutions by Demeter \cite{MR2138874} and Cuny \cites{ MR2122913, MR2122914}, which are carefully constructed perturbations of the Hilbert transform kernel.

Finally, we consider \emph{random} one-sided Hilbert transforms.

\begin{theorem}\label{mainR}
Let $\{ X_n \}$ be  collection of uniformly  bounded, independent,  mean-zero random variables. 
Define 
\[ \HH_X^{\mathbb{Z},*} f := \sup_{ N \geq 1} \left| 
\sum_{n=N}^{\infty} \frac{X_n}{n} f(x-n) \right| .\]
Almost surely, there holds 
\[ \lVert \mathcal H ^{\mathbb Z ,\ast } _X  \;:\; (1,r)\rVert \lesssim_\omega \frac 1 {r-1}, \qquad 1< r < 2. \]
\end{theorem}

By the work of Rosenblatt, and Calder\'{o}n's transference principle, this immediately implies the following corollary.
\begin{cor}
Almost surely, for any $f \in L^1(X)$, for any $(X,\Sigma,\mu,\tau)$ measure preserving system, $\sum_{n=1}^\infty \frac{X_n}{n} \tau^n f$ converges $\mu$-a.e.
\end{cor}
\begin{remark} Specializing $\{ X_n\}$ to be i.i.d.\ $\pm 1$ random signs, we see that almost every choice of $\{ c_n = \pm \frac{1}{n} \}$ provides an affirmative answer to Rosenblatt's Problem \ref{Ros}.
\end{remark}

$\HH_X^{\mathbb{Z},*}$ is too singular to fall under the purview of Theorem \ref{t:a}; in particular, \eqref{e:off} in general fails. Nevertheless, the \emph{approach} used to establish Theorem \ref{t:a} can be suitably modified. For further work in this direction, we refer the reader to the upcoming paper of the first two authors on oscillatory singular integrals.

The structure of our paper is as follows. In $\S 2$, below we will establish our Theorem \ref{t:a} by presenting a general set of conditions for maximally truncated discrete singular integral operators to be bounded on $\ell^r(\Z), 1 < r < \infty$, and to be weakly bounded on $\ell^1(\Z)$. The $\ell^r(\Z)$ theory is familar; the endpoint theory, which is the main novelty of this paper, is motivated by recent work of the first two authors \cites{160908701,160901564}.

In $\S 3$, we will introduce our class of ``admissible" Hardy field functions, and prove pointwise convergence for the associated maximally truncated ergodic singular integral operators on simple functions.

Finally, in $\S 4$, we will prove almost sure sparse bounds for the random one-sided Hilbert transforms.

\subsection{Acknowledgments}
The first author would like to thank Ciprian Demeter for introducing him to random one-sided Hilbert transforms, and to both Ciprian Demeter and Terence Tao for helpful conversations and support.

\subsection{Notation}
As previously mentioned, we use $e(t) := e^{2\pi i t}$. We will let $M_{HL}$ denote the Hardy-Littlewood maximal function acting on the integers. 

We will make use of the modified Vinogradov notation. We use $X \lesssim Y$, or $Y \gtrsim X$ to denote the estimate $X \leq CY$ for an absolute constant $C$. If we need $C$ to depend on a parameter,
we shall indicate this by subscripts, thus for instance $X \lesssim_r Y$ denotes
the estimate $X \leq C_r Y$ for some $C_p$ depending on $r$. We use $X \approx Y$ as
shorthand for $Y \lesssim X \lesssim Y $.

\section{The Maximal Theory}
We present Theorem \ref{t:a}:  Sufficient conditions for a discrete oscillatory operator to satisfy a 
\emph{sparse bound.}

\subsection{The Principle Recursive Step}
By a \emph{dyadic} interval, we mean an interval $ \mathbb Z \cap I$, where $ I$ is a dyadic interval in $ \mathbb R $ of length at least $ 16$.  Set  
\begin{equation}\label{e:TI}
T _{I} f (x) = \mu _{i} \ast (f \mathbf 1_{\frac{1}{3} I} ) (x), \qquad   \lvert  I\rvert = 2 ^{i+3}.   
\end{equation}
Then, $ T_I f$ is supported on $ I$.  We will show the sparse bound for 
\begin{equation}\label{e:T*}
T  ^{\ast} f := 
\sup _{\epsilon } \Bigl\lvert   \sum_{I \;:\; \lvert  I\rvert < \epsilon  } T_I f  \Bigr\rvert.  
\end{equation}
This is sufficient, for this reason. There are a choice of three dyadic grids $ \mathcal D _{s}$, for $ s=1,2,3$ so that 
\begin{equation*}
\mu _{i} \ast f = \sum_{s=1} ^{3} \sum_{I\in \mathcal D _s \;:\; \lvert  I\rvert = 2 ^{i+3} } T_I f .  
\end{equation*}
We fix one such dyadic grid $ \mathcal D = \mathcal D_s$ in what follows.

The definition \eqref{e:T*} is further adapted to different choices of dyadic intervals $ \mathcal I \subset \mathcal D$.  
Set 
\begin{equation}\label{e:T8}
T _{\mathcal I} ^{\ast}  f = 
\sup _{\epsilon } \Bigl\lvert   \sum_{I \in \mathcal I \;:\; \lvert  I\rvert < \epsilon  } T_I f  \Bigr\rvert.  
\end{equation}
On occasion, $ T _{\mathcal I} f$ denotes the full sum above, without a maximal truncation.  
Then, the main Lemma is 

\begin{lemma}\label{l:sparse} Suppose that $ I_0$ is an interval, and $ \mathcal I$ is a collection of intervals $ I\subset I_0$ so that for function $ f $  supported on on $I_0 $ we have 
\begin{equation}\label{e:K}
\sup _{I\in \mathcal I} \langle f \rangle_I \leq K \langle f  \rangle _{I_0},  
\qquad
\sup _{I\in \mathcal I} \langle g \rangle_I \leq K \langle g  \rangle _{I_0},  
\end{equation}
Then, 
\begin{equation}\label{e:sparse}
\langle  T _{\mathcal I} ^{\ast}  f  , g\rangle 
\lesssim \tfrac 1 {r-1} \lvert  I_0\rvert  \langle f \rangle _{I_0}   \langle g \rangle _{I_0,r}  , \qquad 1< r < 2. 
\end{equation}
\end{lemma}

Let us see how this Lemma proves the sparse bound for $ T ^{\ast} $ defined in \eqref{e:T*}, which in turn immediately implies Theorem~\ref{t:a}.  

\begin{theorem}\label{t:t*} Assuming \eqref{diagcorr} and \eqref{e:off}, we have 
\begin{equation}\label{e:T*<}
\lVert T ^{\ast} \;:\; (1,r)\rVert \lesssim \tfrac 1 {r-1}, \qquad 1< r <2 . 
\end{equation}

\end{theorem}

\begin{proof}[Proof assuming Lemma~\ref{l:sparse}]  
We can assume that $ f, g$ are supported on a dyadic interval $ I_0$.  Note that if $ x \in \mathbb Z \setminus I_0$, we have $ T ^{\ast} f (x) \lesssim M _{\textup{HL}} f (x)$. But it is a well known fact that the Hardy-Littlewood maximal function satisfes a $ (1,1)$ sparse bound.

We can therefore take $ I_0$ to be in the sparse collection $ \mathcal S$ that defines our sparse operator, 
and assume that the parent of $ I_0$ in $ \mathcal S$ is at least four times a big.   
It therefore remains to verify the sparse bound 
for $ \langle T ^{\ast} _{\mathcal I_0} f, g \rangle$, where $ \mathcal I_0$ is the collection of all dyadic intervals contained in $ I_0$. 

Let $ \mathcal E $ be the maximal dyadic subintervals $ J$ of $ I_0$ for which $ \langle f \rangle_J \geq 10 \langle f \rangle _{I_0}$ and/or  $ \langle g \rangle_J \geq 10 \langle g \rangle _{I_0}$. 
Setting $ E = \bigcup _{J\in \mathcal E} J$, we see that $ \lvert  E\rvert \leq \frac 1 {5} \lvert  I_0\rvert  $.  
And, letting $ \mathcal I  = \{I\in \mathcal I_0 \;:\; I\not\subset E\}$, we then have 
\begin{align*}
\langle T ^{\ast} _{\mathcal I_0} f, g \rangle
& \leq 
\lvert  \langle T ^{\ast} _{\mathcal I} f, g \rangle\rvert  
+ \sum_{J\in \mathcal E} \lvert  \langle T ^{\ast} _{\mathcal I_0 (J)} f, g \rangle\rvert  
\end{align*}
where $ \mathcal I_0 (J) = \{I\in \mathcal I_0 \;:\; I\subset J\}$.   But the first term 
is controlled by Lemma~\ref{l:sparse}. In particular, the right side of \eqref{e:sparse} is incorporated into the sparse form.  
And the collection $ \mathcal E$ is added to the sparse collection.  The proof follows by recursion.  
\end{proof}

\subsection{Proof of Lemma~\ref{l:sparse}}

We can assume $ \langle f \rangle _{I_0}=1$.  
Let $ \mathcal B$ be the maximal dyadic subintervals $ J\subset I_0$ for which $ \langle f \rangle _{J} \geq K \langle f \rangle _{I_0}$.  Make the  Calder\'{o}n-Zygmund decomposition,  writing $ f = \gamma +b$, 
where 
\begin{align}
b = \sum_{J\in \mathcal B} f \mathbf 1_{J}.  
\end{align}
(Note, no cancellative properties of $ b$ are needed. We can and do take $ f \geq 0$, so that $ b$ is as well.)  
The ``good'' function $ \gamma $ is bounded, so by Proposition~\ref{ellp}, we have 
\begin{align*}
\langle T _{\mathcal I} ^{\ast} \gamma , g\rangle 
\lesssim \lVert  T _{\mathcal I} ^{\ast} \gamma\rVert _{r'} \lVert g\rVert _{r} 
\lesssim \tfrac 1 {r-1} \lvert  I_0\rvert  \langle f \rangle _{I_0} \langle g \rangle _{I_0,r}. 
\end{align*}
It remains to consider the bad function.

For integers $ s$ let $ \mathcal B (s)$ be the intervals $ J\in \mathcal B$ with $ \lvert  J\rvert= 2 ^{s} $, and set 
$ b = \sum_{s=0} ^{\infty } b_s$, where 
\begin{equation}\label{e:bs}
b_s = \sum_{J\in \mathcal B(s)} f \mathbf 1_{J}.  
\end{equation}
The principle points we have 
\begin{equation}\label{e:b}
\lVert b_s\rVert _{\infty } \lesssim 2 ^{s}, \qquad \sum_{s=0} ^{\infty } \lVert b_s\rVert_1 \leq \lvert  I_0\rvert.   
\end{equation}

It is important to note that if $ I\in \mathcal I$, and $ J\in \mathcal B$, if $ I\cap J\neq \emptyset $ we necessarily have $ J\subsetneq I$, just by construction.    So, in particular, 
\begin{equation*}
T_I b = \sum_{s \;:\; 2 ^{s} < \lvert  I\rvert } T _{I} b_s
\end{equation*}
And, therefore 
\begin{equation*}
T _{\mathcal I} b = \sum_{k = 2} ^{\infty } \sum_{0\leq s <  k} T _{\mathcal I (k)} b_{k-s}  .  
\end{equation*}
Above, we are setting $ \mathcal I (k) = \{I\in \mathcal I \;:\; \lvert  I\rvert = 2 ^{k+3} \}$.  
We will hold $ s$ fixed, obtaining geometric decay in that parameter.  Thus, set 
\begin{equation}\label{e:TIs}
T _{I,s} b (x) = T_I b _{k-s}  (x)= 
\mu _{k} \ast (b _{k-s} \mathbf 1_{\frac{1}{3} I} ) (x), \qquad   \lvert  I\rvert = 2 ^{k+3}
\end{equation}
We use the notation $ T _{\mathcal I,s}b$ and $ T _{\mathcal I, s} ^{\ast} b$ in a manner consistent with 
\eqref{e:T8}.  
\begin{lemma}\label{l:s} Under the assumptions of Lemma~\ref{l:sparse}, and the notation \eqref{e:bs} we have 
\begin{equation}\label{e:s}
\lVert  T _{\mathcal I,s} ^{\ast} b \rVert _{r'}  \lesssim  2 ^{-\frac s {3(r-1)}} \lvert  I_0\rvert  ^{\frac {r-1} r}, \qquad 1< r < 2.  
\end{equation}
\end{lemma}

Summing over $ s \geq 1$ will give us the leading constant $ \frac 1 {r-1}$.  

\bigskip 
We take up the proof of  Lemma~\ref{l:s}. 
Note that we have 
\begin{align} \label{e:22}
\lVert T_{I,s} b \rVert_2 ^2 = \sum_{x\in I} b_{k-s} \mathbf 1_{\frac 13 I} (x) \cdot  \tilde \mu _{i} \ast \mu _i \ast (b _{k-s} \mathbf 1_{\frac 13I}) (x). 
\end{align}
Above, $ \lvert  I\rvert = 2 ^{i+3} $.  
Recalling \eqref{diagcorr}, one of our chief assumptions about $ \mu _j$,  the estimate of the right hand side naturally splits into two cases:  
The convolution is dominated by $ \tilde \mu _{i} \ast \mu _i  (x)\mathbf 1_{ x\neq 0}$, the `standard' case, 
or not, the `non-standard' case.

Write 
\[ \mathcal I = \mathcal S_s \cup \mathcal N_s = \bigcup_k \mathcal{S}_s(k) \cup \bigcup_k \mathcal{N}_s(k),\]
where $ I\in \mathcal S_s(k)$ if $ I\in \mathcal I(k)$ and 
\begin{equation}\label{e:standard}
\lVert T_I b_{k-s} \rVert_2 ^2  \leq 64C_0 \lvert  I\rvert ^{-1- \epsilon }  \lVert b_{k-s} \mathbf 1_{I}\rVert_1 ^2 ;   
\end{equation}
and, we collect $I \in \mathcal{N}_s(k)$ if the above inequality \eqref{e:standard} fails.
Here, $ C_0$ is the implied constant in \eqref{diagcorr}, and $ K$ is as in \eqref{e:K}.  

\begin{proposition}\label{p:standard} With the above notation, we have the bound 
\begin{equation}\label{e:standard<}
\lVert   T _{\mathcal S_s ,s} ^{\ast}  b  \rVert _{r'}
 \lesssim   2 ^{- \frac{s}{2(r-1)}  }  \lvert  I_0\rvert ^{\frac {r-1}r}, \qquad 1< r < 2.  
\end{equation}

\end{proposition}

\begin{proof}
For fixed $ k$, we have by \eqref{e:b} 
\begin{equation*}
\lVert  T _{\mathcal S_s (k),s} b \rVert _{\infty } \lesssim 2 ^{-s}.  
\end{equation*}
And, by \eqref{e:standard}, and again \eqref{e:b}, 
\begin{equation*}
\lVert  T _{\mathcal S_s (k),s} b \rVert_{2 } \lesssim 2 ^{-s/2} 2 ^{-k \epsilon /2} \lvert  I_0\rvert ^{1/2}.     
\end{equation*}
Interpolating, and summing over $ k\geq 1$ completes the proof. 
\end{proof}

Thus, the core is the control of the non-standard collections.  In the case that $ I\in \mathcal N_s (k)$, we have 
\begin{equation}\label{e:non}
\lVert T_{I,s} b \rVert_2 ^2   \lesssim \lvert  I\rvert ^{-1}  \lVert b_{k-s} \mathbf 1_{I}\rVert_2 ^2 . 
\end{equation}
since the the convolution with $\tilde \mu _{i} \ast \mu _i   $ in \eqref{e:22} is dominated by 
$  \tilde \mu _{i} \ast \mu _i  (0)\simeq \lvert  I\rvert ^{-1}  $.  
It is worth remarking that the purely $ \ell ^2 $ bound below 
\begin{equation*}
\lVert 
  T _{ \mathcal N_s,s}b   
\rVert _2 \lesssim 2 ^{-s/2} \lvert I_0\rvert ^{1/2} 
\end{equation*}
is known \cites{MR951506,CCPO}, but a method to obtain a bound for maximal truncations is new, and adapted from \cite{160901564}.

The endpoint bound is easy.  

\begin{lemma}\label{l:Ninfty}  Assume that $g$ satisfies \eqref{e:K}
We have the bound uniformly in $s $. 
\begin{equation}\label{e:Ninfty}
\langle T _{\mathcal N_s,s} ^{\ast}  b , g\rangle \lesssim  \lvert I_0\rvert 
\langle f\rangle _{I_0} \langle g\rangle _{I_0}.  
\end{equation}

\end{lemma}

\begin{proof}
 For any choice of measurable functions $ \varepsilon_k \;:\; I \mapsto \{z \;:\; \lvert  z\rvert \leq 1 \}$,  we have 
\begin{align*}
\Bigl\langle 
\sum_{k = s} ^{k_0 }  \varepsilon_k  T _{\mathcal N_s (k)} b_{k-s}, g \Bigr\rangle
&\lesssim  \sum_{k = s} ^{k_0 } 
\sum _{I\in \mathcal{N} _s } \lVert b_{k-s} \mathbf{1}_I \rVert_1 \langle g\rangle _{I}
\\
& \lesssim  \langle g\rangle _{I_0}\sum_{s} \lVert b _{k-s}\rVert_1 \lesssim   
 \lvert I_0\rvert 
\langle f\rangle _{I_0} \langle g\rangle _{I_0}. 
\end{align*}
This proves \eqref{e:Ninfty}.  
\end{proof}

We need a good estimate for maximal truncations at $ \ell ^2 $.

\begin{lemma}\label{l:N2}  We have the bound uniformly in $ $
\begin{equation}\label{e:N2}
\lVert  T _{\mathcal N_s,s} ^{\ast} b\rVert _ 2 \lesssim  2 ^{-s/3} | I_0| ^{1/2} . 
\end{equation}

\end{lemma}

Interpolating between this bound and \eqref{e:Ninfty}  completes the proof of \eqref{e:s}, and hence the proof of Lemma~\ref{l:sparse}. 

\begin{proof}
One more definition is required to address maximal truncations.  For integers $ t $, say that $ I \in \mathcal N _{s,t}$ if $ I\in \mathcal N _{s} (k)$ for some integer $ k$ and 
\begin{equation}\label{e:st}
2 ^{-t+c} \leq \frac { \lVert b _{k-s} \mathbf{1}_I \rVert_1  } {|I|}  < 2 ^{-t+c+1}.  
\end{equation}
Notce that if $\| b _{k-s} \mathbf{1}_I \|_1 \neq 0$, we must have
\[ \frac{\| b_{k-s} \mathbf{1}_I \|_1}{|I|} \gtrsim 2^{-s}\]
for $I \in \mathcal{N}_s(k)$.
Thus, $t \leq s + c$.

We show that 
\begin{equation}\label{e:N3}
\lVert   T _{\mathcal N_{s,t} ,s } ^{\ast}  b\rVert _ 2 \lesssim  s 2 ^{-2s/5} |I_0| ^{1/2} . 
\end{equation}
A summation over $ t \leq s +c $ proves \eqref{e:N2}.

\smallskip 

Crucially, this definition supplies us with a Carleson measure estimate.  For all $ K\in \mathcal N _{s,t} $ we have 
\begin{equation}\label{e:CM}
\sum_{J\in \mathcal N _{s,t} \;:\; J\subset K} \lvert  J\rvert \lesssim 2 ^{t} \lvert  K\rvert.   
\end{equation}
Indeed, if this is not so we have for some large constant $ M$
\begin{align*}
M \cdot 2 ^{t} \lvert  K\rvert & \leq  
\sum_{J\in \mathcal N _{s,t} \;:\; J\subset K} \lvert  J\rvert 
\\
& \leq 2 ^{t+c}  \sum_{k}\sum_{J\in \mathcal N _{s,t} (k) \;:\; J\subset K} \|  b _{k-s} \mathbf{1}_J\|_1 
\lesssim 2 ^{t} \lvert  K\rvert ,  
\end{align*}
by construction, namely \eqref{e:K}, and \eqref{e:b}. This is a contradiction for large, but absolute choice of $ M$. 
Therefore \eqref{e:CM} holds.  

The Carleson measure  condition \eqref{e:CM} says that there are about $ 2 ^{t}$ overlapping intervals $ I \in \mathcal N _{s,t}$.   Set 
\begin{equation*}
F_t = \Bigl\{\sum_{I\in \mathcal N _{s,t}} \mathbf 1_{I } > C 2 ^{t} \Bigr\}. 
\end{equation*}
For a sufficiently large constant $ C$, we have $ \lvert  F_t\rvert < \tfrac 14 \lvert  I_0\rvert  $.
We claim that 
\begin{equation}\label{e:N4}
\lVert   T _{\mathcal N_{s,t} ^{\sharp} } ^{\ast}  b\rVert _ 2 \lesssim  s 2 ^{-2s/5} |I_0|^{1/2} . 
\end{equation}
where $ \mathcal N ^{\sharp} _{s,t} = \{I\in \mathcal N _{s,t} \;:\; I\not\subset F_t\}$.  
To conclude \eqref{e:N3}, we recurse inside the set $ F _{t}$, but this step is easy, and we omit the details. 
It remains to prove \eqref{e:N4}.  

\medskip 
To prove \eqref{e:N4}, we are in a position to apply the Rademacher-Menshov  Lemma~\ref{l:RM} below. It controls the maximal truncations, and its key assumption is an 
an orthogonality condition on the summands.    
To set up the application of this Lemma, we set $ \mathcal M_1$ to be the minimal elements of $ \mathcal N _{s,t} ^{\sharp}$, and inductively set $ \mathcal M _{u+1} $ to be the minimal elements of $ \mathcal N _{s,t} ^{\sharp}  \setminus \bigcup _{v=1} ^{u} \mathcal M _{v}$.  This collection will be empty for $ u > u_0 =C  2 ^{t}$.   Then set 
\begin{equation} \label{e:betadef}
\beta _{u} = \sum_{k} \sum_{I\in \mathcal M _{u} (k)} T _{I} b _{k-s},  
\end{equation}
where $\mathcal{M}_u(k) := \{ I \in \mathcal{M}_u : |I| = 2^{k+3} \}$.
We will show that 
\begin{equation}\label{e:beta}
\Bigl\lVert \sum_{u=1} ^{u_0} \sigma _u \beta _u \Bigr\rVert_2 \lesssim 2 ^{-2s/5} \lvert I_0\rvert ^{1/2}, 
\qquad  \sigma _t \in \{-1,0,1\}.  
\end{equation}
In view of \eqref{e:RM}, we then conclude \eqref{e:N4}, after factoring a $\log u_0 \simeq t \lesssim s $ 
into $ 2 ^{-2s/5}$.  

Now, we have from \eqref{e:non},  \eqref{e:betadef}, \eqref{e:b} and the definition of  $ \mathcal N _{s,t}$,
\begin{align}
\sum_{u=1} ^{u_0}
\lVert \sigma _u  \beta _u \rVert_2 ^2 
& \lesssim \sum_{u=1} ^{u_0} \sum_{k} \sum_{I\in \mathcal M _{u} (k)}  \lvert  I\rvert ^{-1} \lVert b _{k-s} \mathbf 1_{I}\rVert_2 ^2 
\\
& \leq \sum_{u=1} ^{u_0} \sum_{k} \sum_{I\in \mathcal M _{u} (k)}  \lvert  I\rvert ^{-1} \lVert b _{k-s} \mathbf 1_{I}\rVert_{\infty}
\lVert b _{k-s} \mathbf 1_{I}\rVert_1 
\label{e:diagonal}
\\
& \lesssim 2^{-s} \sum_{u=1} ^{u_0} \sum_{k} \sum_{I\in \mathcal M _{u} (k)}  \lVert b _{k-s} \mathbf 1_{I}\rVert_1 
 \label{e:bD}
 \lesssim  2 ^{-s} \lvert  I_0\rvert.   
\end{align}

For $ u  < v$, we have by the off-diagonal assumption \eqref{e:off}, 
\begin{align}
|\langle \sigma _u  \beta _u, \sigma _ v \beta _v \rangle|
& = \Biggl| \sum_{k_v } \sum_{k_u \;:\; k_u  < k_v- c} 
\sum_{J\in \mathcal M _{v} (k_v)}  \sum_{ \substack{I\in \mathcal M _{u} (k_u)\\I\subset J }}   
\langle b _{k_u -s}  , T _{I} ^{\ast} T _{J} b _{k_v -s}  \rangle  \Biggr|
\\
& \lesssim 
 \sum_{k_v } \sum_{k_u \;:\; k_u  < k_v-c} 
\sum_{J\in \mathcal M _{v} (k_v)}  \sum_{ \substack{I\in \mathcal M _{u} (k_u)\\I\subset J }}  
 \lvert  I\rvert ^{- \epsilon } |J|^{-1} \lVert b _{k_u-s} \mathbf 1_{I} \rVert _1 \| b _{k_v-s} \mathbf{1}_J \|_1 
 \\  
 & \lesssim 2 ^{-t} \sum_{k_v} \sum_{k_u \;:\; k_u  < k_v-c} \sum_{J \in \mathcal{M}_v(k_v)} \sum_{ \substack{I\in \mathcal M _{u} (k_u)\\I\subset J }} |I|^{-\epsilon} \frac{|I|}{|J|} \| b_{k_v-s} \mathbf{1}_J\|_1 
 \\&
 \lesssim 2^{-t-\epsilon u} \sum_{k_v} \sum_{J \in \mathcal{M}_{v}(k_v)} \| b_{k_v-s} \mathbf{1}_J \|_1 
 \leq 2^{-t-\epsilon u} \|b\|_1 
 \label{e:bO}
  \lesssim 2 ^{-t- \epsilon u} \lvert  I_0\rvert . 
\end{align}
The last inequalities follow from the construction, and $ \lvert  I\rvert \geq 2 ^{u} $, for $ I\in \mathcal M_u$.  

By Cauchy-Schwarz, we have 
\begin{gather*}
\Bigl\lVert \sum_{u=1} ^{4s/\epsilon} \sigma _u \beta _u \Bigr\rVert_2  
\lesssim s^{1/2} 2^{-s/2} \lvert  I_0\rvert^{1/2},
\\
\Bigl\lVert \sum_{4s /\epsilon }^{u_0} \sigma _u \beta _u \Bigr\rVert_2^2
\lesssim 2 ^{-s} \lvert  I_0\rvert + \sum_{4s/\epsilon \leq u < v \leq u_0} 2^{-t-\epsilon u} |I_0| \lesssim 2 ^{-s} \lvert  I_0\rvert,
\end{gather*}
since $t \leq s + c$.
The proof of \eqref{e:N3} is complete.

\end{proof}

\subsection{Lemmas}

We need this elementary fact, the $ \ell ^r $ bounds. 
\begin{proposition}\label{ellp}
Assuming only \eqref{diagcorr},   the operator  $\HH ^{\ast} _{a} f$ is $\ell^r(\Z)$ bounded, for $ 1< r < \infty $. 
In particular, 
\begin{equation}\label{e:ellr}
\lVert \HH ^{\ast} _{a}  \;:\; \ell ^{r} \mapsto \ell ^{r}\rVert \lesssim \max \{ r, \tfrac 1 {r-1}\} \qquad 1< r < \infty . 
\end{equation}
\end{proposition}

\begin{proof}
We first observe that $\mu_j * \tilde{\mu_j}(0) = \| \mu_j \|_2^2 \lesssim 2^{-j}$; together with the assumed bound on $|\mu_j * \tilde{\mu_j}|$ away from zero, this implies that
\[ \| \mu_j * f\|_2^2 \lesssim 2^{-j} \|f\|_2^2 + 2^{-\epsilon j} \sum_{\Z} |f|(n) M_{HL} f(n) \lesssim 2^{-\epsilon j} \|f\|_2^2.\]
Here, we used that $\mu_j * \tilde{\mu_j}$ is supported in $\{ |x| \lesssim 2^j \}$.  
But we clearly have the $ \ell ^{1} $ and $ \ell ^{\infty }$ estimate below without decay: 
\begin{equation*}
\lVert \mu _j  \;:\; \ell ^{s} \mapsto \ell ^{s} \rVert  \lesssim 1, \qquad s=1, \infty .  
\end{equation*}
Interpolating, we see that $ \lVert \mu _j \;:\; \ell ^{r} \mapsto \ell ^{r}\rVert \lesssim 2 ^{- 2 \epsilon j \frac {r-1}{r} } $ for $ 1< r < 2$. 
Majorizing
\[ \HH f \leq \sum_{j=0}^\infty |\mu_j*f| + M_{HL} f \]
and taking $\ell^r$ norms yields the result. In fact, one may establish a sparse $(r,r)$ bound for $r(\epsilon) < r < 2$ due to the power gain in scale.
\end{proof}

Mentioned above, this is a variant of the Rademacher-Menshov inequality that we used to control maximal 
truncations.   This has been observed many times. See \cite{MR2403711}*{Theorem 10.6}.

\begin{lemma}\label{l:RM}  Let $ (X, \mu )$ be a measure space, and $ \{ \phi _j \;:\; 1\leq j \leq N\}$ a sequence of functions 
which satisfy the Bessel type inequality below, for all sequences of coefficents $c_j \in \{ 0, \pm 1\}$, 
\begin{equation}\label{e:bessel}
\Bigl\lVert \sum_{j=1} ^{N}  c_j \phi _j \Bigr\rVert _{L ^2 (X)} \leq A .  
\end{equation}
Then, there holds 
\begin{equation}\label{e:RM}
\Bigl\lVert\sup _{1< n \leq N} 
\Bigl\lvert 
 \sum_{j=1} ^{n}   \phi _j
\Bigr\rvert
\Bigr\rVert _{L ^2 (X)} \lesssim A   \log(2+ N) .  
\end{equation}

\end{lemma}

\section{Specializing to Hardy Fields}
We now introduce Hardy fields and some of their properties. We refer the reader to e.g.\ \cite{14074736} and the references contained therein for further discussion of Hardy field functions and their applications to ergodic theory; in particular, the following introduction to Hardy field functions is taken from \cite{14074736}*{\S 2}.

We call two real valued functions of one real variable that are continuous for large values of $s \in \RR$ \emph{equivalent} if they coincide for large $s\in\RR$. We say that a property holds for large $s$ (or eventually) if it holds for every $s$ in an interval of the form $[s_0,\infty)$. The equivalence classes under this relation are called \emph{germs}. The set of all germs we denote by $B$ which is a ring.

\begin{definition}
A \emph{Hardy field} is a subfield of $B$ which is closed under differentiation. The union of all Hardy fields is denoted by $\U$.
\end{definition}
One can show that $\mathcal{U}$ contains the class $\mathcal{L}$ of logarithmico-exponential functions of Hardy, i.e., the class of functions which can be obtained by finitely many combinations of real constants, the variable $s$, $\log$, $\exp$, summation and multiplication. Thus, for example, it contains functions of the form $s^\alpha = \exp( \alpha \log s)$, $\alpha\in\RR$.

Another property of Hardy fields is that each Hardy field is totally ordered with respect to the order $<_{\infty}$ defined by
$$
f  <_{\infty} g \quad \Longleftrightarrow \quad f(s)< g(s) \quad \text{for all large } s.
$$
Since the class $\mathcal{L}$ belongs to every maximal Hardy field, we conclude that every element of $\U$ is comparable to every logarithmico-exponential  function. In particular, we can define the \emph{type} of a function $p\in U$ to be
$$
t(p):=\inf\{\alpha\in \RR:\, |p(s)|< s^\alpha \text{ for large }s\}.
$$
We say that $p$ is \emph{subpolynomial} if $t(p)<+\infty$, i.e., if $|p|$ is dominated by some polynomial. In particular, for eventually positive subpolynomial $p$ with finite type there is $\alpha\in \RR$ such that for every $\eta$ there is an $s_0$ so that \[ s^{\alpha-\eta}<p(s)<s^{\alpha+\eta}\]
holds for every $s>s_0$. Note that considering eventually positive $p$ is not  a restriction since every nonzero $p\in\U$ is either eventually positive or eventually negative.

We now consider subpolynomial elements of $\U$ with positive non-integer type, such as for example $p(s)=5s^\pi+s\log s.$ More precisely, we introduce the following classes.

\begin{definition}
For $\delta\in(0,1/2)$, $M\geq 1$ and $m \geq 1$ denote by $\n_{\delta, M, m}$ the set of all $p\in \U$ so that there exist $\alpha \in [\delta, 1-\delta]$ and $\eta \ll_{\delta,m} 1$ with
\begin{equation}\label{M_delta,M,m}
\frac{1}{M}s^{m+\alpha-\eta-j}\leq p^{(j)}(s)\leq M s^{m+\alpha+\eta-j} \text{ for all }s\geq 1 \text{ and } j=0,\ldots,m+2.
\end{equation}
We say that $p$ is ``admissible" if it is in some class $\n_{\delta,M,m}$.
\end{definition}
\begin{remark}
By l'H\^{o}pital's rule, the condition $s^{m+\alpha - \eta} \lesssim p(s) \lesssim s^{m+\alpha + \eta}$ is enough to guarantee that
\[ s^{m+\alpha -j - \eta} \lesssim_j p^{(j)}(s) \lesssim_j s^{m+\alpha -j + \eta} \]
for each $j$. We simply choose to make the implicit constant uniform over the first $m+2$ derivatives.
\end{remark}

We have the following lemma.

\begin{lemma}\label{check}
For admissible $p$, the measures
\[ \mu_j^p := \sum_{2^{j-1} < n \leq 2^j} \frac{e(p(n))}{n} \delta_n\]
satisfy the hypotheses of Theorem~\ref{t:a}, namely \eqref{diagcorr} and \eqref{e:off}. 
\end{lemma}
\begin{remark}
As we will see from the proof, the $\epsilon > 0$ of gain will depend only on $\delta,M,m$.
\end{remark}

For pointwise convergence reasons, we will need the following technical complement to Lemma \ref{check}:
\begin{lemma}\label{check1}
For any $\kappa > 0$,
\[ \mu_{j,\kappa}^p := \sum_{(1+\kappa)^{j-1} < n \leq (1+\kappa)^j} \frac{e(p(n))}{n} \delta_n\]
satisfies
\begin{equation}\label{e:technical}
|\mu_{j,\kappa}^p * \widetilde{ \mu_{j,\kappa}^p  }| \lesssim (1+\kappa)^{-j} \delta_0 + (1+\kappa)^{-(1+\epsilon)j}.
\end{equation}
\end{lemma}

We defer the proofs of our technical lemmas to the following subsection, and complete the proof of pointwise convergence now.

\begin{proof}[Proof of Theorem \ref{main}, assuming Lemmas \ref{check} and \ref{check1}]
Let $p$ be an admissible Hardy field function, thus $p \in \n_{\delta,M,m}$ for some $\delta, M, m$.
By Lemma \ref{check},  Theorem~\ref{t:a}, and the Calder\'{o}n transference principle \cite{MR0227354}, we know that the maximal function
\[ \HH_p^*f := \sup_N \left| \sum_{n=N}^\infty \frac{e(p(n))}{n} \tau^n f \right|\]
is weakly bounded on $L^r(X), \ 1 \leq r < \infty$. By a standard density argument, it therefore suffices to prove pointwise convergence for simple (bounded) functions, $g$. 
In fact, by \cite{MR1325697}*{Lemma 1.5}, it suffices to prove only that for each simple $g$, and each $\kappa > 1$, the limit
\[ \lim_{ j \to \infty } \sum_{ (1+\kappa)^j < n} \frac{e(p(n))}{n} \tau^n g = 0 \qquad  \mu-\text{a.e.}\]
But this is straightforward.  The technical estimate \eqref{e:technical} implies 
\[ \sigma (g,j) := 
\Bigl\lVert \sum_{ (1+\kappa)^j < n \leq (1+\kappa)^{j+1} } \frac{e(p(n))}{n} \tau^n g  \Bigr\rVert
\lesssim (1+\kappa)^{-\epsilon j} \| g\|_2,\]
and thus
\begin{align*}
\Bigl\lVert  \sum_{(1+\kappa)^j < n} \frac{e(p(n))}{n} \tau^n g \Bigr\rVert_2 &\leq \sum_{ l =j}^\infty  \sigma (g,l) \lesssim (1+\kappa)^{-\epsilon j/2} \|g\|_2. 
\end{align*}
Consequently,
\[ \sum_{j =0 }^\infty \left| \sum_{ (1+\kappa)^{j-1} < n } \frac{e(p(n))}{n} \tau^n g \right|^2 \]
is an integrable function, which proves our claim. 
\end{proof}

\subsection{The Proof of Lemmas \ref{check} and \ref{check1}}
In what follows, we will need the following result of van der Corput which
appears in \cite{MR1545013}*{Satz 4}
\begin{lemma} \label{l:vdc}
  Let $k \geq 2$ be an integer and put $K = 2^k$. Suppose that $a \le b
  \le a + N$ and that $f:[a, b] \to \mathbb R$ has continuous $k$th
  derivative that satisfies the inequality $0 < \lambda \leq
  |f^{(k)}(x)| \leq h\lambda$ for all $x \in [a, b]$.

  Then
  \begin{equation}
   \left| \sum_{a \le n \le b} e (f(n))  \right| \lesssim hN\big( \lambda^{1/(K - 2)} +
    N^{-2/K} + (N^k\lambda)^{-2/K} \big).\label{eq:14}
  \end{equation}
\end{lemma}

With this tool in hand, we are prepared for the proof of our technical lemmas.
\begin{proof}[Proof of Lemma \ref{check}]
We begin with the case where $k < j$, and seek to prove that
\[ \left| \sum_{ 2^{k-1}<n\leq 2^k, 2^{j-1}<n+x\leq 2^j} \frac{ e(p(n+x) - p(n))}{(n+x)n} \right| \lesssim 2^{-\epsilon k -j }.\]
Set
\[ A := \max\{ 2^{k-1}, 2^{j-1} - x \}, \ B := \min\{ 2^{k}, 2^{j} - x \}.\]
Note that $B-A \lesssim 2^k$.
By summation by parts, it suffices to show that
\[ \max_{ A< K \leq B} \left| \sum_{A < n \leq K} e(p(n+x) - p(n)) \right| \lesssim 2^{(1-\epsilon)k}.\]
We begin with the case $m = 1$. Note that we may assume that $K -A \gtrsim 2^{(1-\epsilon)k}$.

For large enough $k$, the phase $f(n) = f_x(n) := p(n+x) - p(n)$ has second derivative
\[ (2^k)^{\alpha - 1 - \eta} \lesssim_M |f''(n)| \lesssim_M (2^k)^{\alpha -1 + \eta} \]
for $A < n \leq B$.
By Lemma \ref{l:vdc},
\[ \max_{A < K \leq B} \left| \sum_{A < n \leq K} e(f(n)) \right| \lesssim 2^{(1+2\eta - \epsilon)k},\]
for some $\epsilon = \epsilon(\alpha)$ bounded away from zero. 

We next turn to the second part of the lemma, where we consider diagonal interactions, $j = k$, evaluated at $ x \neq 0$.

In this case, by the mean-value theorem, the phase $f(n)$ has second derivative
\[ |x| (2^j)^{\alpha - 2 -\eta} \lesssim_M |f''(n)| \lesssim_M |x| (2^j)^{\alpha -2 + \eta}.\]
By Lemma \ref{l:vdc},
\[ \left| \sum_{A < n \leq K} e(f(n)) \right| \lesssim 2^{(1+2\eta - \epsilon)j},\]
for some $\epsilon = \epsilon(\alpha)$ bounded away from zero.

The $m \geq 2$ cases follow similarly from the $(m+1)$th Van der Corput lemma.
\end{proof}

The proof of Lemma \ref{check1} is similar; the details are left to the reader.

\section{The Random One-sided Hilbert Transform}
The goal of this section is to prove Theorem \ref{mainR}, reproduced below for the reader's convenience. First, we recall the $\{X_n\}$:
bounded, independent, and mean-zero random variables on a probability space $\Omega$. Then:

\begin{theorem} 
Let $\{ X_n \}$ be  collection of uniformly  bounded, independent,  mean-zero random variables. 
Define 
\[ \HH_X^{\mathbb{Z},*} f := \sup_{ N \geq 1} \left| 
\sum_{n=N}^{\infty} \frac{X_n}{n} f(x-n) \right| .\]
Almost surely, there holds 
\[ \lVert \mathcal H ^{\mathbb Z ,\ast } _X  \;:\; (1,r)\rVert \lesssim_\omega \frac 1 {r-1}, \qquad 1< r < 2. \]
\end{theorem}
Analogous to the previous sections, we let
\[ \mu_i = \mu_i^{\omega} := \sum_{2^{i-1} < n \leq 2^i} \frac{X_n}{n} \delta_n;\]
we will generally suppress the $\omega$ in our notation. Our first order of business is to establish good estimates on the convolutions $\tilde{\mu_i}* \mu_j, \ i \leq j$. We do so in the following subsection.

\subsection{Random Preliminaries}
We need the following well known large deviation inequality.

\begin{lemma}[Chernoff's Inequality]\label{CHERN}
Let $\{Z_n\}$ be mean-zero, independent random variables, all of which are almost surely bounded in magnitude by $1$. Then there exists an absolute constant $c  > 0$ so
\[ \PP \biggl( \biggr| \sum_{n=1}^N Z_n \biggr| \geq A \biggr) \lesssim \max\{ e^{-c \frac{A^2}{V_N}}, e^{-cA} \},\]
where $V_N = \sum_{n=1}^N \mathbb{E}|Z_n|^2$.
\end{lemma}

Using Lemma \ref{CHERN}, we have the following control over $\tilde{\mu_i} * \mu_j, \ i \leq j$.

\begin{lemma} \label{l:mu}
Almost surely, the following hold.
\begin{itemize}
\item $| \tilde{\mu_i}* \mu_i (x)| \lesssim_\omega 2^{-5i/4}$ for $x \neq 0$;
\item For $i < j$, $| \tilde{\mu_i}* \mu_j  | \lesssim_\omega \frac{\sqrt{j}}{2^{i/2 +j}} \mathbf{1}_{[0, 2^j)}$, and thus for all $j \lesssim 2^{i/2}$,
\[ | \tilde{\mu_i}* \mu_j | \lesssim_\omega 2^{-i/4 -j} \mathbf{1}_{[0, 2^j)};
\]
\item For $2^{i/2} \ll j$, 
\begin{equation}
\label{e:Z} 
 Z_{i,j} := \{ |\tilde{\mu_i}* \mu_j| \gg_\omega 2^{-i/4 -j} \} \subset \bigcup_{m} I_m 
\end{equation}
where each $|I_m| = 2^i$ is dyadic, and the (disjoint) union is over at most a constant multiple of $e^{-c_0 2^{i/2}  } \times 2^{j-i}$ intervals, for some (small) absolute constant $c_0$. 
\end{itemize}
\end{lemma}
\begin{remark}
We will use without comment the trivial upper bound (useful on the exceptional sets $Z_{i,j}$):
$|\tilde{\mu_i}*\mu_j| \lesssim 2^{-j}$.
\end{remark}
\begin{proof}
The convolution in question is explicitly, 
\[ \tilde{\mu_i}*\mu_j(x) = \sum_{2^{i-1} < n \leq 2^i, \ 2^{j-1} < x+n \leq 2^j} \frac{X_{x+n}X_n}{(x+n)n},  \qquad i \leq j.  \]
Observe that the sum above is over uniformly bounded, mean zero random variables.  And that 
\begin{equation}
\mathbb E 
\tilde{\mu_i}*\mu_j(x) ^2 
\lesssim 2^{-i-2j} , \qquad x\neq 0
\end{equation}
where the implied constant depends upon the uniform bound on the random variables $\{X_n\}$. Since the sum in the definition of $\tilde{\mu_i}*\mu_j(x)$ can  be separated into two sums, each over independent mean-zero random variables, the first point follows from Chernoff's inequality, Lemma \ref{CHERN}, and a Borel-Cantelli argument.
 The second point is similar.

For the third, for dyadic $|I| = 2^i$, $I \subset [0,2^j)$, consider the random variables
\[ V_I^{i,j} := \sup_{x \in I} | \tilde{\mu_i}*\mu_j (x) |.\]
By subdividing into (say) ten subfamilies, we may assume that the $\{ V_I^{i,j} : I \}$ are independent. By Chernoff's inequality, we know that for each $I$,
\[ \mathbb{P}( V_I^{i,j} \geq 2^{-i/4 -j} ) \lesssim e^{-c 2^{i/2}} \]
for some absolute $c > 0$.
By a binomial distribution argument, provided that $c_0$ is sufficiently small, there exists some $c > c' > 0$ so that
\[ \PP( |\{ I : V_I^{i,j} \geq 2^{-i/4 -j}\}| \geq e^{-c_0 2^{i/2}} 2^{j-i} ) 
\lesssim e^{ - c' 2^{j - i/2} e^{-c_0 2^{i/2}} }.\]
Summing over $j \gg 2^{i/2}$ and applying Borel-Cantelli yields the result.
\end{proof}

\subsection{The Proof}
We follow the argument of $\S 2$, with the obvious notational changes; the key lemma needed is analogous to 
\eqref{e:N4} in the proof of 
Lemma \ref{l:N2}. We refer to $\S 2$ for relevant definitions.

\begin{lemma}\label{l:RN2}  There exists some $\delta > 0$ so that we have the bound uniformly in $s$
\begin{equation}\label{e:RN2}
\lVert   T _{\mathcal N_{s,t} ^{\sharp} } ^{\ast}  b\rVert _ 2 \lesssim  2 ^{-\delta s} |I_0|^{1/2} . 
\end{equation}
\end{lemma}
\begin{proof}
As in the proof of Lemma \ref{l:N2}, set
\[ \beta_u := \sum_k \sum_{I \in \mathcal{M}_u(k)} T_I b_{k-s}; \]
we will show that
\[ \Bigl\| \sum_{u=1}^{u_0} \sigma_u \beta_u \Bigr\|_2 \lesssim 2^{-\delta's} |I_0|^{1/2}, \qquad \sigma_t \in \{-1,0,1\}, \]
where $u_0 \lesssim 2^t \lesssim 2^s$. This will allow us to conclude Lemma \ref{l:RN2}, by an application of the Rademacher-Menshov Lemma~\ref{l:RM}.
The diagonal terms are estimated exactly as in  \eqref{e:diagonal}. By Cauchy-Schwarz, we need only show that 
\begin{equation}\label{e:doublesum} \sum_{ Cs \leq u < v \leq u_0 } | \langle \sigma_u \beta_u, \sigma_v \beta_v \rangle| \lesssim 2^{-\delta' s} |I_0|. \end{equation}
Divide the summands on the left into 
\begin{align}
|\langle \sigma _u  \beta _u, \sigma _ v \beta _v \rangle| 
& \leq \Biggl| \sum_{k_v } \sum_{k_u \;:\; k_u  < k_v} 
\sum_{J\in \mathcal M _{v} (k_v)}  \sum_{ \substack{I\in \mathcal M _{u} (k_u)\\I\subset J }}   
\langle b _{k_u -s}  , T _{I} ^{\ast} T _{J} b _{k_v -s}  \rangle  \Biggr|
\\
& \lesssim \text{Main}(u,v) + \text{Exceptional}(u,v), \end{align}
where the inner products on the right are treated according to the estimates of Lemma \ref{l:mu}. The two terms are 
\begin{gather} 
\text{Main}(u,v) := 
 \sum_{k_v } 
\sum_{J\in \mathcal M _{v} (k_v)} 
\sum_{k_u \;:\; k_u  < k_v}   \sum_{ \substack{I\in \mathcal M _{u} (k_u)\\I\subset J }}  
 \lvert  I\rvert ^{- 1/4 } |J|^{-1} \lVert b _{k_u-s} \mathbf 1_{I} \rVert _1 \| b _{k_v-s} \mathbf{1}_J \|_1,
 \\
 \begin{split} 
&\text{Exceptional}(u,v) \\
&  := 
 \sum_{k_v } 
 \sum_{J\in \mathcal M _{v} (k_v)}
 \sum_{\substack{k_u  \\ 2^{k_u/2}  \ll k_v} } 
\sum_{y \in \frac 13 J} b_{k_v-s}(y) 
  \left(\sum_x   \sum_{ \substack{I\in \mathcal M _{u} (k_u)\\I\subset J }} b_{k_u-s} \mathbf{1}_{\frac{1}{3}I}(x) \mathbf{1}_{Z_{k_u,k_v}}(x-y) \right).  
  \end{split}
\end{gather}
 
We begin by estimating $\text{Main}(u,v)$ as follows:
\begin{align}
\text{Main}(u,v) & 
\lesssim 2 ^{-t} \sum_{k_v} \sum_{J \in \mathcal{M}_v(k_v)} \sum_{k_u \;:\; k_u  < k_v} \sum_{ \substack{I\in \mathcal M _{u} (k_u)\\I\subset J }} |I|^{-1/4} \frac{|I|}{|J|} \| b_{k_v-s} \mathbf{1}_J\|_1 \\
&\lesssim 2^{-t-u/4} \sum_{k_v} \sum_{J \in \mathcal{M}_{v}(k_v)} \| b_{k_v-s} \mathbf{1}_J \|_1 \\
&\lesssim 2^{-t-u/4} \|b \|_1 
\lesssim 2^{-t-u/4}|I_0|. \end{align}
As for $\text{Exceptional}(u,v)$, the key estimate is \eqref{e:Z}: 
\begin{align}
&\text{Exceptional}(u,v) \\
& \qquad \lesssim
  \sum_{k_v}  
\sum_{J\in \mathcal M _{v} (k_v)} \langle b_{k_v-s}\rangle _{\frac{1}{3}J}   \sum_{k_u \;:\; 2^{k_u/2}  \ll k_v} \sup_{y \in \frac{1}{3}J}
\biggl\| \sum_{\substack{ I\in \mathcal M _{u} (k_u) \\  I\subset J }} b_{k_u-s} \mathbf{1}_{\frac{1}{3}I} \mathbf{1}_{y+Z_{k_u,k_v}} \biggr\|_1 
\\
& \qquad \lesssim  
  \sum_{k_v }  
\sum_{J\in \mathcal M _{v} (k_v)} \langle b_{k_v-s}\rangle _{\frac{1}{3}J}  \sum_{k_u \;:\; 2^{k_u/2}  \ll k_v} 2^{k_u-s} \sup_{y \in \frac{1}{3}J} \biggl| \bigcup_{ I \in \mathcal{M}_u(k_u)} \{ I \subset J, I \cap (y+Z_{k_u,k_v}) \neq \emptyset \}\biggr| \\
& \qquad \lesssim \sum_{k_v}
\sum_{J\in \mathcal M _{v} (k_v)} \langle b_{k_v-s}\rangle _{\frac{1}{3}J}  \sum_{k_u \;:\; 2^{k_u/2}  \ll k_v} 2^{k_u-s} 2^{k_v} e^{-c_0 2^{k_u/2}} \\
& \qquad \lesssim 2^{-s} \sum_{k_v} 
\sum_{J\in \mathcal M _{v} (k_v)}  \int_J  b_{k_v-s}   \sum_{k_u \;:\; 2^{k_u/2}  \ll k_v} 2^{k_u} e^{-c_0 2^{k_u/2}} 
\\
& \qquad \lesssim 2^{-s}2^u e^{-c_0 2^{u/2}} \| b\|_1 
 \lesssim 2^{-s-u}|I_0|.
\end{align}
In treating $\text{Exceptional}(u,v)$, the key estimates that we used were that $\| b_k\|_{\infty} \lesssim 2^k$, see \eqref{e:b}, and  this consequence of \eqref{e:Z}: For any $y$,
\[ |\{ I : |I| = 2^{k_u} \ \text{dyadic}, I \cap Z_{k_u,k_v} + y \neq \emptyset\}| \lesssim 2^{k_v} e^{-c_0 2^{k_u/2}}. \]
The upshot is that we have majorized $|\langle \sigma_u \beta_u, \sigma_v \beta_v\rangle|$ by $2 ^{-t- u/4} \lvert  I_0\rvert$.
Summing this over 
\[ Cs \leq u < v \leq u_0 \lesssim 2^t\] 
yields the desired bound, completing the proof.
\end{proof}
 
\bibliographystyle{alpha}	


\end{document}